\documentclass{amsart}
\usepackage{amssymb}
\usepackage{hyperref}
\usepackage{url}
\usepackage[all]{xy}
\usepackage{tikz}
\usepackage[none]{hyphenat}
\usepackage{comment}
\usepackage[OT2,OT1]{fontenc}

\DeclareRobustCommand\cyr{%
  \renewcommand\rmdefault{wncyr}%
  \renewcommand\sfdefault{wncyss}%
  \renewcommand\encodingdefault{OT2}%
  \normalfont
  \selectfont}
\DeclareTextFontCommand{\textcyr}{\cyr}

\usetikzlibrary{matrix,arrows,decorations.pathmorphing}

\newtheorem{thm}{Theorem}

\newtheorem{lem}[thm]{Lemma}
\newtheorem{cor}[thm]{Corollary}

\newtheorem{prop}[thm]{Proposition}

\theoremstyle{definition}
\newtheorem{defn}[thm]{Definition}

\newtheorem{rem}[thm]{Remark}

\newcommand{\RN}[1]{%
  \textup{\uppercase\expandafter{\romannumeral#1}}%
}

\numberwithin{thm}{section}

\def\Q{\mathbf{Q}}
\def\F{\mathbf{F}}

\def\Fp{\F_p}
\def\Ftwo{\F_2}

\def\cP{\mathcal{P}}

\def\cP{\mathcal{P}}

    \DeclareFontFamily{U}{wncy}{}
    \DeclareFontShape{U}{wncy}{m}{n}{<->wncyr10}{}
    \DeclareSymbolFont{mcy}{U}{wncy}{m}{n}
    \DeclareMathSymbol{\Sh}{\mathord}{mcy}{"58}

\def\q{\mathfrak{q}}

\def\d{\mathfrak{d}}

\def\Hom{\mathrm{Hom}}
\def\Gal{\mathrm{Gal}}

\def\rk{\mathrm{rk}}

\def\Sel{\mathrm{Sel}}

\def\Frob{\mathrm{Frob}}

\def\res{\mathrm{res}}

\def\dimp{\dim_{\Fp}}
\def\dimtwo{\dim_{\Ftwo}}

\def\too{\longrightarrow}

\def\dirsum#1{\underset{#1}{\textstyle\bigoplus}}

\def\d2{\dim_{\Ftwo}}

\title[Large Shafarevich-Tate groups]
    {Large Shafarevich-Tate groups over quadratic number fields}
\author{Myungjun Yu}
\address{Department of Mathematics,
University of Michigan, Ann Arbor, MI 48109-1043,
USA}

\email{\href{mailto:myungjuy@umich.edu}{myungjuy@umich.edu}}

\begin{document}

\begin{abstract}
Let $E$ be an elliptic curve over the rational field $\Q$. Let $K$ be a quadratic extension over $\Q$. We show that (under mild conditions on $E$) for every $r>0$, there are infinitely many quadratic twists $E^d/\Q$ of $E/\Q$ such that $\dimtwo(\Sh(E^d/K)[2]) > r$.
\end{abstract}

\maketitle
\sloppy

\section{Introduction}
  Let $E$ be an elliptic curve over a number field $F$. The group of rational points $E(F)$ is known to be a finitely generated abelian group and its rank is called the \emph{Mordell-Weil rank}, denoted by $\rk(E(F))$. The Mordell-Weil rank is one of the most interesting and misterious invariants in the study of elliptic curves. In general, it is very difficult to compute the Mordell-Weil rank, and we often study various \emph{Selmer groups}, which give upper bounds for the Mordell-Weil rank. These Selmer groups also contain the information of another very important object, called the \emph{Shafarevich-Tate group}. Precisely, one can define the Shafarevich-Tate group of $E/F$ as follows.
$$
\Sh(E/F) := \ker\left(H^1(F, E(\overline{F})) \to \prod_v H^1(F_v, E(\overline{F_v}))\right),
$$
where $v$ varies over the places of $F$. 

  It is conjectured that $\Sh(E/F)$ is finite, but one may wonder whether the order of $\Sh(E/F)$ can be arbitrarily large. Since it is often easier to study its $p$-torsion part (for a prime $p$) instead, a natural question one can ask is: ``can $\Sh(E/F)[p]$ be arbitrarily large?"

 Cassels \cite{Cassels} first showed the unboundedness of $\dim_{\F_{3}}(\Sh(E/\Q)[3])$ as $E$ varies over elliptic curves over $\Q$. Showing the unboundedness of $\dim_{\Fp}(\Sh(E/\Q)[p])$ has been successful for some primes including $2,3,5,7$ and $13$. For related results, see \cite{2}, \cite{3}, \cite{4}, and \cite{5}. They use the fact that there are infinitely many elliptic curves over $\Q$ with rational $p$-isogenies, so this idea doesn't seem to work for all primes $p$. 
 
 If $F$ is allowed to vary (with bounded degree over $\Q$), Kloosterman \cite{Kloosterman} showed that for any $p$, $\dimp(\Sh(E/F)[p])$ is unbounded (as $E$ also varies). Clark and Sharif \cite{ClarkSharif} proved that for any elliptic curve $E/\Q$, $\dimp(\Sh(E/F)[p])$ is unbounded when $F$ varies over degree $p$ (not necessarily Galois) extensions over $\Q$.

 If $F$ is a fixed Galois extension of degree $p$ over $\Q$, Matsuno \cite{Matsuno} proved that  $\dimp(\Sh(E/F)[p])$ is unbounded as  $E$ varies over elliptic curves over $\Q$. It is important to note that the unboundedness of $\dimp(\Sh(E/\Q)[p])$ does not imply that of $\dimp(\Sh(E/F)[p])$ since the kernel of the restriction map $$\Sh(E/\Q)[p] \to \Sh(E/F)[p]$$ could be also large in size under the condition that $[F : \Q] = p$. We give a further discussion of Matsuno's result when $p=2$ in Remark \ref{Matsunop2}.

 Now let  $D$ be a squrefree integer and let $K = \Q(\sqrt{D})$ be a quadratic number field. For a squarefree integer $d$, we write $E^d$ for the quadratic twist of $E$ by $\Q(\sqrt{d})/\Q$. In this paper, we improve Matsuno's theorem when $p=2$ by showing $\dimtwo(\Sh(E^d/K)[2])$ is unbounded as $E^d$ varies over quadratic twists of a fixed elliptic curve $E/\Q$ (under mild conditions on $E$). 
      
More precisely, the main theorem of this paper is 
\begin{thm}
\label{main}
Let $E$ be an elliptic curve over $\Q$ with no non-trivial rational $2$-torsion point. Let $\Delta$ be the discriminant of a model of $E$. If $K \neq \Q(\sqrt{\Delta})$, then for every $r > 0$, there exist infinitely many quadratic twists $E^d$ of $E$ over $\Q$ such that 
$$
\dimtwo\left(\Sh(E^d/K)[2]\right) - \dimtwo\left(\Sh(E^d/\Q)[2]\right) > r.
$$
\end{thm}

As an immediate consequence, we have
\begin{cor}
\label{coro}
Let $E/\Q$, $K$, and $r$ be as in the previous theorem. Then there exist infinitely many quadratic twists $E^d$ of $E$ over $\Q$ such that
$$
\dimtwo\left(\Sh(E^d/K)[2]\right) > r.
$$ 
\end{cor}

\begin{rem}
In an appendix to \cite{HL}, Rohrlich proved that for every $r > 0$, there exists $E^d/\Q$ such that 
$$
\dimtwo\left(\Sh(E^d/\Q)[2]\right) > r. 
$$
Note that the above corollary is not a direct consequence of Rohrlich's theorem since the kernel of the restriction map
$$
\Sh(E^d/\Q)[2] \to \Sh(E^d/K)[2]
$$
can be also large as mentioned above. 
\end{rem}

\begin{rem}
\label{Matsunop2}
What Matsuno \cite{Matsuno} proved in the case $p=2$ is as follows. Let $A/\Q$ be an elliptic curve defined by the equation
$$
y^2 + xy = x^3 + 8mx^2 + lmx,
$$
with certain restrictions on the prime divisors of $l$ and $m$. Then Corollary \ref{coro} holds when $E$ is replaced by $A$. 
\end{rem}

\subsection{Strategy of the proof} 
For a number field $F$ and an elliptic curve $E/F$, we have a short exact sequence
\begin{equation}
\label{2exact}
0 \too E(F)/2E(F) \too \Sel_2(E/F) \too \Sh(E/F)[2] \too 0. 
\end{equation}
It follows that
\begin{equation}
\label{dimsel}
\dimtwo\left(\Sel_2(E/F)\right) = \rk\left(E(F)\right) + \dimtwo\left(E(F)[2]\right) + \dimtwo\left(\Sh(E/F)[2]\right). 
\end{equation}

We will construct $E^d$ so that the following conditions are all satisfied. 
\begin{enumerate}
\item
$\rk\left(E^d(\Q)\right) = \rk\left(E^d(K)\right)$,
\item
$\dimtwo\left(\Sel_2(E^d/\Q)\right) < a$ for a fixed constant $a$,
\item
$\dimtwo\left(\Sel_2(E^d/K)\right) \gg 0$ .  
\end{enumerate}
Then, by \eqref{dimsel}, (i) and (ii), we have  $\rk(E^d(K)) < a$ and $\dimtwo(\Sh(E^d/\Q)[2]) < a$. Therefore it follows from \eqref{dimsel} and (iii) that
$$
\dimtwo\left(\Sh(E^d/K)[2]\right) - \dimtwo\left(\Sh(E^d/\Q)[2]\right)  \gg 0.
$$


\section{Selmer ranks over $K$}
Let $E$ be an elliptic curve over $\Q$ with no non-trivial rational $2$-torsion point. We write $q$ for a prime of $\Q$. Let $\Q_q$ denote the $q$-adic completion of $\Q$. 

\begin{defn}
We define
$$W_{q,K} := \ker\left(H^1(\Q_q, E(\overline{\Q_q}))    \to    H^1(K_\q, E(\overline{\Q_q})\right),$$ 
where $\q$ is a prime of $K$ above $q$ and $K_\q$ is the completion of $K$ at $\q$. 
\end{defn}

\begin{lem}
\label{Wv}
Let $q$ be an odd prime such that $\dimtwo\left(E(\Q_q)[2]\right) = 2$. Suppose $E$ has additive reduction at $q$ and suppose $q$ is inert in $K/\Q$. Then 
\begin{enumerate}
\item
$E(K_\q)[2^\infty] = E(K_\q)[2] = E[2]$ and  
\item
$\dimtwo(W_{q, K}) = 2$. 
\end{enumerate}
\end{lem}

\begin{proof}
It follows from \cite[Proposition $\RN{7}$.5.4(a)]{Silverman} that $E$ has additive reduction over $K_\q$ . By \cite[Theorem $\RN{7}$.6.1]{Silverman}, we have $[E(K_\q) : E_0(K_\q)] \le 4.$  Since $E_0(K_\q)$ is divisible by $2$, (i) follows. We now show (ii). By the inflation-restriction sequence, we have 
$$W_{q, K} \cong H^1\left(K_\q/\Q_q, E(K_\q)\right) \cong H^1\left(K_\q/\Q_q, E(K_\q)[2^\infty]\right) \cong \Hom\left(K_\q/\Q_q, E[2]\right),$$ 
where the second isomorphism follows from decomposing $E(K_\q)$ into the direct sum of $2$-primary part and non $2$-primary part (note that $[K_\q:\Q_q] = 2$). Then (ii) easily follows. 
\end{proof}

Define $A_E$ to be the set of primes $q$ of $\Q$ satisfying the conditions of Lemma \ref{Wv}. We write $\Sel_2(E/K)$ for the classical $2$-Selmer group of $E/K$ (see Definition \ref{selmer}).

\begin{prop}
\label{SeloverK}
We have
$$
\dimtwo\left(\Sel_2(E/K)\right) \ge 2 |A_E|. 
$$
\end{prop}

\begin{proof}
The proposition immediately follows from the remark after \cite[Corollary 3.3]{Matsuno} and Lemma \ref{Wv}(ii). 
\end{proof}


\section{Selmer ranks over $\Q$}
We continue to assume that $E/\Q$ has no non-trivial rational $2$-torsion point. Let $d$ be a squarefree integer.

\begin{lem}
Let $q$ be an odd prime. Then 
$$
\dimtwo\left(E(\Q_q)/2E(\Q_q)\right) = \dimtwo\left(E(\Q_q)[2]\right). 
$$
\end{lem}

\begin{proof}
See for example, \cite[Lemma 2.2(i)]{MRhilbert}.
\end{proof}

\begin{defn}
\label{cP}
Let $G_E$ be the set of odd primes where $E$ has good reduction. For $0 \le i \le 2$, define 
$$
\cP_{E,i} = G_E \cap \left\{q : \dimtwo\left(E(\Q_q)[2]\right) = i \right\}.   
$$ 
\end{defn}

\begin{rem}
If $q \in G_E$, then  $\Q(E[2])$ is an unramified extension of $\Q$ at $q$, where $\Q(E[2])$ is the smallest extension of $\Q$ that contains the coordinates of all points of $E[2]$. Write $\Frob_q$ for the Frobenius automorphism at $\q$ in $\Gal(\Q(E[2])/\Q)$. Then by \cite[Lemma 4.3]{KMR}, we have
\begin{enumerate}
\item
$\Frob_q$ has degree $3$ if and only if $q \in \cP_{E,0}$,
\item
$\Frob_q$ has degree $2$ if and only if $q \in \cP_{E,1}$,
\item
$\Frob_q = 1$ if and only if $q \in \cP_{E,2}$.
\end{enumerate}
\end{rem}

\begin{defn}
Let $v$ be a place of $\Q$. For $d_v \in \Q_v^\times/(\Q_v^\times)^2$, define the image 
$$\alpha_v(d_v) := \mathrm{Im}\left(E^{d_v}(\Q_v)/2E^{d_v}(\Q_v) \to H^1(\Q_v, E^{d_v}[2]) \cong H^1(\Q_v, E[2])\right),$$ 
where the first map is the Kummer map for multiplication by $2$ on $E^{d_v}$ and the second map is induced by the canonical isomorphsim $E^{d_v}[2] \cong E[2]$.
\end{defn}

\begin{rem}
\label{onlydependon}
For $d \in \Q^\times/(\Q^\times)^2$, note that $\alpha_v(d)$ only depends on the image of $d$ in $\Q_v^\times/(\Q_v^\times)^2$. 
\end{rem}

We recall the Tate local duality:

\begin{thm}
\label{tld}
The Weil pairing and cup product induce a nondegenerate pairing
\begin{equation}
\label{tlp}
\langle  \text{ , } \rangle_v : H^1(\Q_v, E[2]) \times H^1(\Q_v, E[2]) \too H^2(\Q_v, \{\pm1\}) \cong \Ftwo.
\end{equation}
\end{thm}

\begin{proof}
See \cite[Theorem 7.2.6]{cohomology}.
\end{proof}

Define the restriction map at $v$
  $$\res_v : H^1(\Q,E[2]) \too H^1(\Q_v, E[2]).$$

\begin{defn}
\label{selmer}
The $2$-Selmer group $\Sel_2(E^d/\Q) \subset H^1(\Q,E[2])$ is the (finite)
$\F_2$-vector space defined by the following exact sequence
$$
0 \too \Sel_2(E^d/\Q) \too H^1(\Q,E[2]) \too \prod_{v} H^1(\Q_v,E[2])/\alpha_v(d),
$$
where the rightmost map is the sum of $\res_v$. In particular, if $d =1$, it is the classical $2$-Selmer group of $E/\Q$. 
\end{defn}

We define various Selmer groups as follows.

\begin{defn}
\label{variousselmer}
Let $S$ be a finite set of places of $\Q$. Define
$$
\Sel_{2, S}(E^d/\Q) := \left\{x \in \Sel_{2}(E^d/\Q) : \mathrm{res}_{q}(x) = 0 \text{ for all $q \in S$} \right\}.
$$
Define
$$
\Sel_{2}^S(E^d/\Q) := \left\{x \in H^1(\Q, E[2]) : \mathrm{res}_{q}(x) \in \alpha_q(d) \text{ for all $q \not\in S$} \right\}.
$$
If $S = \{v\}$, we simply write $\Sel_{2, v}(E^d/\Q)$ and $\Sel_{2}^v(E^d/\Q)$ for $\Sel_{2, \{v\}}(E^d/\Q)$ and $\Sel_{2}^{\{v\}}(E^d/\Q)$, respectively.
\end{defn}

\begin{thm}
Let $S$ be a finite set of places of $\Q$. 
\label{ptd}
The images of right hand restriction maps of the following exact sequences are orthogonal complements with respect to the pairing given by the sum of pairings \eqref{tlp} of the places $v\in S$
$$
\xymatrix@R=3pt@C=7pt{
0 \ar[r] & \Sel_{2}(E/\Q) \ar[r] & \Sel_2^S(E/\Q) \ar[r]
    & \dirsum{v \in S} H^1(\Q_v,E[2])/\alpha_v(1),  \\
0 \ar[r] & \Sel_{2, S}(E/\Q) \ar[r] & \Sel_2(E/\Q) \ar[r] & \dirsum{v \in S} \alpha_v(1).
}
$$
In particular, 
$$
\dimtwo(\Sel_2^S(E/\Q)) - \dimtwo(\Sel_{2, S}(E/\Q)) = \sum_{v \in S} \dimtwo(\alpha_v(1)).
$$
\end{thm}

\begin{proof}
See \cite[Theorem 2.3.4]{MRkolyvagin}. 
\end{proof}

\begin{cor}
\label{corofptd}
Let $T$ be a set of places containing all places $v$, where the local conditions of $E$ and $E^d$ are not the same, i.e., $\alpha_v(1) \neq \alpha_v(d)$. Then
$$
\left|\dimtwo(\Sel_2(E/\Q))  - \dimtwo(\Sel_2(E^d/\Q))\right| \le \sum_{v \in T} \dimtwo(\alpha_v(1)).
$$
\end{cor}

\begin{proof}
We have 
$$
\Sel_{2,T}(E/\Q) \subseteq \Sel_2(E/\Q), \Sel_2(E^d/\Q) \subseteq \Sel_2^T(E/\Q). 
$$
Then the corollary is an immediate consequence of Theorem \ref{ptd}.
\end{proof}

\begin{lem}
\label{ram}
For $0 \le i \le 2$, let $q$ be an odd prime in $\cP_{E,i}$. Let $v_q$ be the normalized (additive) valuation of $\Q_q$. For $d_q \in \Q_q^\times$, if $v_q(d_q)$ is odd, then
$$
\alpha_q(1) \cap \alpha_q(d_q) = 0.
$$  
\end{lem}

\begin{proof}
Note that the condition $v_q(d_q)$ is odd is equivalent to the condition that $\Q_q(\sqrt{d_q})$ is a ramified (quadratic) extension of $\Q_q$. The lemma follows from \cite[Lemma 2.11]{MRhilbert}.
\end{proof}

For $A \subset H^1(\Q,E[2])$, we write $\res_q(A)$ for the image of $A$ under $\res_q$.

\begin{lem}
\label{selmerbound}
Let $q$ be an odd prime in $\cP_{E,2}$. There exists $d_q \in \Q_q^\times$ such that $v_q(d_q)$ is odd and
$$
\dimtwo\left(\alpha_q(d_q) \cap \res_q(\Sel_2^q(E/\Q))\right) \le \dimtwo\left(\alpha_q(1) \cap \res_q(\Sel_2^q(E/\Q))\right) .
$$
\end{lem}

\begin{proof}
For such $q$, we have $\dimtwo(\res_q(\Sel_2^q(E/\Q))) = 2$ by Theorem \ref{ptd}. For simplicity, let 
$$X:=\alpha_q(1) \cap \res_q(\Sel_2^q(E/\Q)).$$  

If $\dimtwo(X) =2$, there is nothing to show. If $\dimtwo(X) = 1$, then for any $d_q \in K_q^\times$, we have $\dimtwo\left(\alpha_q(d_q) \cap \res_q(\Sel_2^q(E/\Q))\right) = 1$ by Lemma \ref{ram}, and \cite[Theorem 2.5 and Lemma 2.9]{Yu2}. If $X = 0$, choose $c_q$, $d_q \in \Q_q^\times$ so that $\Q_q(\sqrt{c_q})$ and $\Q_q(\sqrt{d_q})$ are distinct ramified extension over $\Q_q$. By \cite[Lemma 3.9 and Proposition 3.10(i)]{Yu3}, without loss of generality, we can assume
\begin{align*}
\alpha_q(c_q) &= \res_q\left(\Sel_2^q(E/\Q)\right), \\
\alpha_q(d_q) & \cap \res_q\left(\Sel_2^q(E/\Q)\right) = 0,
\end{align*}
so the lemma follows. 
\end{proof}

\begin{prop}
\label{Ewithmanyaddred}
Assume that $E/\Q$ has no rational $2$-torsion point. Let $q_1$,$q_2$,$\ldots$, $q_s$ be elements of $\cP_{E,2}$. Then there exist infinitely many squarefree odd integers $d$ such that the following conditions are satisfied. 
\begin{enumerate}
\item
$\dim_{\Ftwo} \left(\Sel_2(E^d/\Q)\right) \le \dim_{\Ftwo} \left(\Sel_2(E/\Q)\right)$,
\item
Every prime divisor of $d$ is in $\cP_{E,0} \cup \cP_{E,2}$ (in particular if $q|d$, $E$ has good reduction at $q$)
\item
The prime factors of $d$ in $\cP_{E,2}$ are exactly $q_1, q_2,\ldots, q_s$. 
 \end{enumerate}
\end{prop}

\begin{proof}
By induction, it is enough to find $d$ so that (i),(ii) hold and $\cP_{E,2}$ contains only one prime that divides $d$. This is a consequence of Lemma \ref{selmerbound} and \cite[Lemma 6.4\footnote{Typo : the conclusion of \cite[Lemma 6.4]{Yu1} should be `$\Sel_2(J^\chi/K) = \Sel_2(J, \psi_\ell)$'    not `$\Sel_2(J^\chi/K) = \Sel_2(J, \chi_\ell)$'}]{Yu1} (and its proof). 
\end{proof}

\begin{lem}
\label{additive}
Suppose $E/\Q$ has a good reduction at an odd prime $q$. If $q | d$, then $E^d$ has additive reduction at $q$. 
\end{lem}

\begin{proof}
See, for example, \cite[Proposition $\RN{7}$.5.1(c)]{Silverman}.
\end{proof}

\section{Proof of Theorem \ref{main}} 
Recall $K = \Q(\sqrt{D})$ for a squarefree integer $D$. First, we fix an integer $c$ such that 
$$\dimtwo\left(\Sel_2(E/\Q)\right) \le c.$$ 
By the argument of subsection 1.1, it is enough to find (infinitely many) $E^d$ such that (i),(ii), and (iii) in susbsection 1.1 are satisfied. Since $K \neq \Q(\sqrt{\Delta})$, We have that 
$$\Q(E[2]) \cap K = \Q.$$
Then the Chebotarev density theorem implies that for any $s>0$, we can choose $q_1, q_2, \ldots, q_s \in \cP_{E,2}$ so that $q_i$'s are inert in $K/\Q$.  In virtue of Proposition \ref{Ewithmanyaddred} and Lemma \ref{additive}, we may assume $E$ (by replacing it with some quadratic twist) satisfies
\begin{itemize}
\item
$|A_E| \gg 0$,
\item
$\dimtwo\left(\Sel_2(E/\Q)\right) \le c.$
\end{itemize}

 Let $S$ be a set of primes of $\Q$ containing $2$, primes where $E$ has bad reduction, and divisors of $D$. By \cite[Theorem]{HL} (for $E^D$) and a Kolyvagin's theorem \cite{Koly}, there are infinitely many squarefree odd integers $d$ satisfying
    \begin{itemize}
    \item
    $d$ has at most $4$ prime factors,
    \item
    $\rk(E^{Dd}(\Q)) = 0,$ and
    \item
    For all $q \in S$, $d \in (\Q_q^\times)^2$.   
    \end{itemize}
Then $$\rk(E^d(K)) = \rk(E^d(\Q)) + \rk(E^{dD}(\Q)) = \rk(E^d(\Q)),$$ which is (i) in subsection 1.1. Note that the last condition on the choice of $d$ implies that $E$ and $E^d$ have the same local conditions except at the (prime) divisors of $d$ and the archimedean place (by Remark \ref{onlydependon} and \cite[Lemma 2.6]{Yu2}). Then Corollary \ref{corofptd} proves that 
$$
\dimtwo\left(\Sel_2(E^d/\Q)\right) \le c + 10,
$$
so (ii) in subsection 1.1 holds. Also, by the choice of $d$, it is clear that $A_{E} \subseteq A_{E^d}$. Finally, (iii) in subsection 1.1 follows from Proposition \ref{SeloverK}.

\begin{rem}
In the proof of Theorem \ref{main}, we constructed $E^d$ with a large number of primes where $E^d$ has additive reduction. In particular, $d$ has many prime divisors. It would be interesting to know whether it is possible to find $E^d$ with a large (2-torsion part of) Shafarevich-Tate group over $K$ with a small number of prime divisors of $d$. 
\end{rem}

\bibliographystyle{abbrv}
\bibliography{Sharef}

\begin{thebibliography}{10}

\bibitem{2}
R.~B\"olling.
\newblock Die {O}rdnung der {S}chafarewitsch-{T}ate-{G}ruppe kann beliebig
  gro\ss \ werden.
\newblock {\em Math. Nachr.}, 67:157--179, 1975.

\bibitem{Cassels}
J.~W.~S. Cassels.
\newblock Arithmetic on curves of genus {$1$}. {VI}. {T}he {T}ate-\v safarevi\v
  c group can be arbitrarily large.
\newblock {\em J. Reine Angew. Math.}, 214/215:65--70, 1964.

\bibitem{ClarkSharif}
P.~L. Clark and S.~Sharif.
\newblock Period, index and potential. {III}.
\newblock {\em Algebra Number Theory}, 4(2):151--174, 2010.

\bibitem{3}
T.~Fisher.
\newblock Some examples of 5 and 7 descent for elliptic curves over {$\bf Q$}.
\newblock {\em J. Eur. Math. Soc. (JEMS)}, 3(2):169--201, 2001.

\bibitem{HL}
J.~Hoffstein and W.~Luo.
\newblock Nonvanishing of {$L$}-series and the combinatorial sieve.
\newblock {\em Math. Res. Lett.}, 4(2-3):435--444, 1997.
\newblock With an appendix by David E. Rohrlich.

\bibitem{KMR}
Z.~Klagsbrun, B.~Mazur, and K.~Rubin.
\newblock Disparity in {S}elmer ranks of quadratic twists of elliptic curves.
\newblock {\em Ann. of Math. (2)}, 178(1):287--320, 2013.

\bibitem{Kloosterman}
R.~Kloosterman.
\newblock The {$p$}-part of the {T}ate-{S}hafarevich groups of elliptic curves
  can be arbitrarily large.
\newblock {\em J. Th\'eor. Nombres Bordeaux}, 17(3):787--800, 2005.

\bibitem{Koly}
V.~A. Kolyvagin.
\newblock Finiteness of {$E({\bf Q})$} and {\cyr {sh}}{$(E,{\bf Q})$} for a
  subclass of {W}eil curves.
\newblock {\em Izv. Akad. Nauk SSSR Ser. Mat.}, 52(3):522--540, 670--671, 1988.

\bibitem{4}
K.~Kramer.
\newblock A family of semistable elliptic curves with large
  {T}ate-{S}hafarevitch groups.
\newblock {\em Proc. Amer. Math. Soc.}, 89(3):379--386, 1983.

\bibitem{5}
K.~Matsuno.
\newblock Construction of elliptic curves with large {I}wasawa
  {$\lambda$}-invariants and large {T}ate-{S}hafarevich groups.
\newblock {\em Manuscripta Math.}, 122(3):289--304, 2007.

\bibitem{Matsuno}
K.~Matsuno.
\newblock Elliptic curves with large {T}ate-{S}hafarevich groups over a number
  field.
\newblock {\em Math. Res. Lett.}, 16(3):449--461, 2009.

\bibitem{MRkolyvagin}
B.~Mazur and K.~Rubin.
\newblock Kolyvagin systems.
\newblock {\em Mem. Amer. Math. Soc.}, 168(799):viii+96, 2004.

\bibitem{MRhilbert}
B.~Mazur and K.~Rubin.
\newblock Ranks of twists of elliptic curves and {H}ilbert's tenth problem.
\newblock {\em Invent. Math.}, 181(3):541--575, 2010.

\bibitem{cohomology}
J.~Neukirch, A.~Schmidt, and K.~Wingberg.
\newblock {\em Cohomology of number fields}, volume 323 of {\em Grundlehren der
  Mathematischen Wissenschaften [Fundamental Principles of Mathematical
  Sciences]}.
\newblock Springer-Verlag, Berlin, second edition, 2008.

\bibitem{Silverman}
J.~H. Silverman.
\newblock {\em The arithmetic of elliptic curves}, volume 106 of {\em Graduate
  Texts in Mathematics}.
\newblock Springer-Verlag, New York, 1986.

\bibitem{Yu3}
M.~Yu.
\newblock 2-{S}elmer near-companion curves.
\newblock 2016.
\newblock https://arxiv.org/abs/1610.01195.

\bibitem{Yu1}
M.~Yu.
\newblock Selmer ranks of twists of hyperelliptic curves and superelliptic
  curves.
\newblock {\em J. Number Theory}, 160:148--185, 2016.

\bibitem{Yu2}
M.~Yu.
\newblock On 2-{S}elmer ranks of quadratic twists of elliptic curves.
\newblock {\em Math. Res. Lett.}, 24(5):1565--1583, 2017.

\end{thebibliography}

\end{document}